\documentclass{amsart}
\usepackage{amsmath,amsthm}
\usepackage{amsfonts,amssymb}
\usepackage{accents}
\usepackage{enumerate}
\usepackage{accents,color}
\usepackage{graphicx}

\newcommand{\lvt}{\left|\kern-1.35pt\left|\kern-1.3pt\left|}
\newcommand{\rvt}{\right|\kern-1.3pt\right|\kern-1.35pt\right|}

\newtheorem{thm}{Theorem}[section]
\newtheorem{cor}[thm]{Corollary}
\newtheorem{lem}[thm]{Lemma}

\newtheorem{defn}[thm]{Definition}

\theoremstyle{remark}
\newtheorem{rem}{Remark}[section]

 \def\la{{\langle}}
 \def\ra{{\rangle}}

 \def\ff{{\mathfrak f}} 
 \def\hf{{\mathfrak h}}

 \def\a{{\alpha}}
 \def\b{{\beta}}
 \def\g{{\gamma}}
 
 \def\t{{\theta}}
 \def\l{{\lambda}}
 \def\d{{\delta}}

 \def\la{{\langle}}
 \def\ra{{\rangle}}

 \def\CJ{{\mathcal J}}

 \def\CS{{\mathcal S}}

 \def\CV{{\mathcal V}}
 
 \def\CW{{\mathcal W}}

 \def\NN{{\mathbb N}}

 \def\RR{{\mathbb R}}

\def\lla{\langle{\kern-2.5pt}\langle}      
\def\rra{\rangle{\kern-2.5pt}\rangle}

\newcommand{\wh}{\widehat}

\def\f{\frac}

\graphicspath{{./}}
\begin{document}

\title[]
{Approximation by polynomials in Sobolev spaces with Jacobi weight}
\author{Yuan Xu}
\address{  Department of Mathematics\\ University of Oregon\\
    Eugene, Oregon 97403-1222.}
\email{yuan@math.uoregon.edu}
 
\date{\today}
\keywords{Approximation, simultaneous approximation, Sobolev space, Jacobi weight}
\subjclass[2010]{ 41A10, 41A25, 42C05, 42C10, 33C45}
\thanks{The author was supported in part by NSF Grant DMS-1510296}
 
\begin{abstract}
Polynomial approximation is studied in the Sobolev space $W_p^r(w_{\alpha,\beta})$ that consists of functions 
whose $r$-th derivatives are in weighted $L^p$ space with the Jacobi weight function $w_{\alpha,\beta}$. This 
requires simultaneous approximation of a function and its consecutive derivatives up to $s$-th order with $s \le r$. 
We provide sharp error estimates given in terms of $E_n(f^{(r)})_{L^p(w_{\alpha,\beta})}$, the error of best 
approximation to $f^{(r)}$ by polynomials in $L^p(w_{\alpha,\beta})$, and an explicit construction of the polynomials
that approximate simultaneously with the sharp error estimates. 
\end{abstract}

\maketitle

\section{Introduction} \label{sect1}
\setcounter{equation}{0}

Polynomial approximation on a finite interval is a classical problem at the center of approximation theory. The purpose
of this paper is to consider simultaneous approximation of a function and its derivatives by polynomials on an interval 
in $L^p$ norms defined with respect to a Jacobi weight function. Although this problem has been studied by several 
researchers, our results are new in several aspects. 

Let $w_{\a,\b}$ be the Jacobi weight function defined by $w_{\a,\b}(x) : = (1-x)^\a (1+x)^\b$ for $\a,\b > -1$ and 
$x \in (-1,1)$. For $1 \le p < \infty$, define 
$$
\|f\|_{L^p(w_{\a,\b})} := \left( \int_{-1}^1 |f(x)|^p w_{\a,\b}(x) dx\right)^{1/p},
$$
and, for $p = \infty$, define this norm as the usual uniform norm $\|f\|_\infty$. For $r \in \NN$, let $C^r[-1,1]$ denote the 
space of functions that have $r$-th continuous derivatives on $[-1,1]$.  For $1\le p < \infty$, let $\CW_p^r(w_{\a,\b})$ 
be the Sobolev space 
$$
      W_p^r(w_{\a,\b}): = \{f \in C^{r-1}[-1,1]: f^{(r)} \in L^p(w_{\a,\b})\},
$$
and, for $p  = \infty$, define this space as $C^r[-1,1]$. We define the norm of $W_p^r(w_{\a,\b})$ by
$$
  \|f\|_{W_p^r(w_{\a,\b})} : =  \left( \sum_{k=0}^r \|f^{(k)}\|_{L^p(w_{\a,\b})}^p \right)^{1/p}.
$$
For $n \in \NN$, let $\Pi_n$ denote the space of polynomials of degree at most $n$ in one variable. The standard 
error of best approximation by polynomials in $\Pi_n$ is defined by
$$
  E_n(f)_{L^p(w_{\a,\b})} : = \inf_{p \in \Pi_n} \|f - p\|_{L^p(w_{\a,\b})}. 
$$
The characterization of this quantity via an appropriate modulus of smoothness lies in the center of Approximation 
Theory and is widely studied; see, for example, \cite{DL, DT}. For $p=2$, the $n$-th partial sum $S_n^{\a,\b} f$
of the Fourier-Jacobi series satisfies 
\begin{equation}\label{eq:least}
   E_n(f)_{L^2(w_{\a,\b})} = \|f-S_n^{\a,\b} f\|_{L^2(w_{\a,\b})}. 
\end{equation}
However, using $S_n^{\a,\b} f$ for approximation in $W_2^r (w_{\a,\b})$ gives a much weaker result than optimal 
(cf. \cite{CQ, Guo}), which will be discussed in Section 2 below.  

Throughout this paper, we denote by $c$ a generic constant, independent of $n$, whose value may vary from 
line by line. We prove two types of results for approximation in the Sobolev space. 

\begin{thm}\label{thm:1.1}
Let $\a,\b > -1$. Assume $f\in W^s_p(w_{\a,\b})$ if $1 \le p < \infty$ or $f \in C^s[-1,1]$ if $ p =\infty$. Then there exists 
a polynomial $p_n \in \Pi_{2n}$ such that 
\begin{equation}\label{eq:thm1.1}
  \|f - p_n\|_{W_p^s(w_{\a,\b})} \le c\, E_n(f^{(s)})_{L^p(w_{\a,\b})}, \quad 1\le p \le \infty.
\end{equation}
\end{thm}

\begin{thm}\label{thm:1.2}
Let $\a,\b > -1$. Assume $f\in W^s_p(w_{\a,\b})$ if $1 \le p < \infty$ or $f \in C^s[-1,1]$ if $ p =\infty$. Then there exists a polynomial $p_n \in \Pi_{2n}$ such that 
\begin{equation}\label{eq:thm1.2}
  \|f^{(k)} - p_n^{(k)}\|_{L^p(w_{\a,\b})} \le c\, n^{-s+k} E_n(f^{(s)})_{L^p(w_{\a,\b})}, \quad 1\le p \le \infty,
\end{equation}
provided either $\a =0$ or $\b =0$. 
\end{thm}

Evidently, the estimate \eqref{eq:thm1.2} is stronger than the estimate \eqref{eq:thm1.1} but it holds under more 
restrictive conditions. Moreover, \eqref{eq:thm1.2} is sharp; in fact, the order of the estimate is sharp for each 
fixed $k$. 

The estimate \eqref{eq:thm1.1} provides a sharp estimate for the error of best polynomial 
approximation in the Sobolev norm. It is known \cite{MT} that, for $r \in \NN$, 
\begin{equation}\label{eq:Jackson}
  E_n(f)_{L^p(w_{\a,\b})} \le c\, n^{-r} \|\phi^r f^{(r)}\|_{L^p(w_{\a,\b})}, \qquad \phi(x):=\sqrt{1-x^2},
\end{equation}
so that the righthand side of \eqref{eq:thm1.1} and \eqref{eq:thm1.2} can be stated with $E_n(f^{(s)})_{L^p(w_{\a,\b})}$ 
replaced by $n^{-r+s} \|f^{(r)}\|_{L^p(w_{\a,\b})}$. In the case $p=2$, estimates in the form 
\begin{equation} \label{eq:old}
   \|f - p_n \|_{W_2^s(w_{\a,\b})} \le c n^{-r+s} \|f^{(r)}\|_{L^2(w_{\a,\b})}, \qquad 0 \le s \le r,
\end{equation}
have been established and used in the spectral method for numerical solution of differential equations in some cases; 
more precisely, such an estimate was first established in \cite{CHQZ, CQ} for the case of Chebyshev and Legendre 
polynomials ($\a=\b = -1/2$ or $0$) when $s =1$, and later established for general $\a,\b$ and $s=1$ by several 
researchers, see \cite{GW, STW} and references therein. In latter works, the norm in the lefthand side of \eqref{eq:old} 
is often replaced by 
$$
  \|f \|_{W_2^s(w_{\a,\b})}^*:=  \left( \sum_{k=0}^s \|f^{(k)}\|_{L_2(w_{\a+k,\b+k})}^2 \right)^{1/2}
$$
and the norm in the righthand side is replaced by $\|f^{(r)}\|_{L^2(w_{\a+r,\b+r})}$, which we call $*$-version. 
By \eqref{eq:Jackson}, our estimates can be stated in terms of the norm of $f^{(r)}$ for $r \ge s$, so that \eqref{eq:thm1.1}
is stronger than \eqref{eq:old} and offers an estimate somewhat different from \eqref{eq:old} in $*$-version. 

The estimate \eqref{eq:thm1.2} is what is known as simultaneous approximation in the Approximation Theory community,
where it is a folklore that each increased derivative reduces the order of approximation by $n^{-1}$. However, such 
estimates are usually established with a modulus of smoothness of $f^{(r)}$ in place of $E_n(f^{(r)})_{L^p(w_{\a,\b})}$ 
in \eqref{eq:thm1.2} in the literature (cf. \cite{Kilgore, K}). 

The polynomial $p_n$ in the theorems can be expressed in a simple explicit formula. For $p=2$, it is the $n$-th partial 
sum operator of the Fourier series in the Sobolev orthogonal polynomials in $W_2^s(w_{\a,\b})$, which are polynomials
that are orthogonal with respect to an inner product that involves derivatives. It should be pointed out, however, that this 
fact does not follow from the usual Hilbert space argument, since the norm of $W_2^s(w_{\a,\b})$ is not arising from the 
square root of the inner product that defines the Sobolev orthogonality. Sobolev orthogonal polynomials have been 
studied extensively in the special function community (cf. \cite{MX} and the references therein), but not their orthogonal 
series, and in spectral method community, often with zero boundary at the end of the intervals. These polynomials are 
usually given in terms of the Jacobi polynomials with negative integer parameters, which requires appropriate extensions 
that could be rather delicate (cf. \cite{AAR02, GW, APP98, LX, X16}). We shall give a more direct definition of the 
Sobolev orthogonal polynomials in $W_2^s(w_{\a,\b})$ that does not require such extensions.  

The paper is organized as follows. In the following section, we discuss the approximation behavior of the $L^2$ partial 
sum operator $S_n^{\a,\b} f$ in the Sobolev space, which gives suboptimal result. In Section 3, we consider a Sobolev
inner product in $W_2^r(w_{\a,b})$ and define a family of its orthogonal polynomials, for all $\a,\b > -1$, in an elegant
formula that is more suitable for studying orthogonal series in terms of them. Approximation by polynomials or orthogonal 
series are studied in the following two sections. In particular, more elaborate version of Theorems \ref{eq:thm1.1} and 
\ref{eq:thm1.2} will be established in Section 4 and Section 5, respectively.  

\section{Jacobi polynomials and Fourier Jacobi series} \label{sect2}
\setcounter{equation}{0}

For $\a,\b>-1$, the Jacobi polynomials are defined by \cite[(4.21.2)]{Sz},
\begin{equation}\label{eq:JacobiP}
  P_n^{(\a,\b)}(t) = \frac{(\a+1)_n}{n!} {}_2F_1 \left(\begin{matrix} -n, n+\a + \b+1\\ \a+1\end{matrix}; \frac{1-t}{2} \right)
\end{equation}
in terms of the hypergeometric function ${}_2F_1$. For convenience, we shall define 
\begin{equation*}
  J^{\alpha,\beta}_{n}(t) =  \frac{2^n}{(n+\a+\beta+1)_{n}}{P}^{(\alpha,\beta)}_{n}(t).
\end{equation*}
One advantage of this normalization is the following identity, by \cite[(4.5.5)]{Sz},  
\begin{align} \label{eq:diffJ}
  \frac{d}{dt} J_n^{\a,\b}(t) =  J_{n-1}^{\a+1,\b+1}(t). 
\end{align}
These polynomials are orthogonal with respect to the inner product
$$
  \la f,g\ra_{\a,\b} := \int_{-1}^1 f(x) g(x) w_{\a,\b}(x) dx. 
$$
The Fourier orthogonal expansion of $f$ in $L^2(w_{\a,\b})$ is defined by 
$$
   f = \sum_{n=0}^\infty \wh f_n^{\a,\b}  J_n^{\a,\b}, \quad\hbox{where}\quad 
      \wh f_n^{\a,\b} = \frac{\la f, J_n^{\a,\b}\ra_{\a,\b}}{h_n^{\a,\b}}
$$
and $h_n^{\a,\b}:= \la J_n^{\a,\b}, J_n^{\a,\b}\ra_{\a,\b}$ is given by \cite[(4.3.3)]{Sz}
\begin{equation} \label{eq:h_n}
  h_n^{\a,\b} =  \frac{2^{\alpha+\beta+1}}{2n+\alpha+\beta+1} 
    \frac{\Gamma(n+\alpha+1)\,\Gamma(n+\beta+1)\Gamma(n+\alpha+\beta+1)}{n!\,\Gamma(2n+\alpha+\beta+1)^2}.
\end{equation}
The $n$-th partial sum of this expansion, defined by 
$$
  S_n^{\a,\b} f(x) = \sum_{k=0}^n \wh f_k^{\a,\b} J_k^{\a,\b}(x),
$$
is the least square polynomial of degree $n$, that is, \eqref{eq:least} holds. The operator $S_n^{\a,\b}$ can be written 
as a linear integral operator. 

For approximation in $L^p(w_{\a,\b})$, $p\ne 2$, we can define a near best approximation 
operator as follows. We call $\eta$ an admissible function if $\eta$ is a $C^\infty$ function on $\RR_+$ satisfying 
$\eta(t) = 1$ for $0\le t \le 1$ and $\eta(t) =0$ for $t \ge 2$. For an admissible $\eta$, define 
$$
    V_n^{\a,\b} f: = \sum_{k=0}^{2n} \eta\left(\frac{k}{n} \right) \wh f_k^{\a,\b} J_k^{\a,\b}.
$$
It is well known that $V_{n}^{\a,\b}$ defines a bounded linear operator in $L^p(w_{\a,\b})$ for $1 \le p \le \infty$ and 
it preserves polynomials up to degree $n$, that is, $V_{n}^{\a,b} f = f$ for $f \in \Pi_n$; consequently, the following 
theorem holds (see, for example, \cite{X05})

\begin{thm} 
Let $\a,\b > -1$. For $f\in L^p(w_{\a,\b})$ if $1 \le p < \infty$, or $f \in C[-1,1]$ if $p = \infty$, 
\begin{equation} \label{eq:Jacobi-near-best}
     \|f -  V_{n}^{\a,\b} f \|_{L^p(w_{\a,\b})} \le c \, E_n(f)_{L^p(w_{\a,\b})}, \qquad 1 \le p \le \infty.
\end{equation}
\end{thm}

We will also need the Jackson type estimate for the error of best approximation. Let 
$\phi(x):= \sqrt{1-x^2}$. The following theorem was established in \cite{MT}. 

\begin{thm} \label{thm:Jackson}
Let $\a,\b > -1$. For $f \in W_p^r(w_{\a,\b})$ if $1 \le p < \infty$, or $f \in C^r[-1,1]$ if $p = \infty$,  
\begin{equation} \label{eq:Jackson2}
   E_n(f)_{L^p(w_{\a,\b})} \le c\, n^{-r} \|\phi^r f^{(r)}\|_{L^p(w_{\a,\b})}, \qquad 1 \le p \le \infty. 
\end{equation}
\end{thm} 

In the rest of this section, we consider the approximation behavior of $S_n^{\a,\b} f$ in $L^2(w_{\a,\b})$ and in
$W_2^r(w_{\a,\b})$. Some of the results below are no doubt known but they provide contrast to our latter development 
and our proof is simple. We start with a lemma that is suggestive for our later study. Let $\partial$ denote the differential
operator. 

\begin{lem}
Let $\a,\b >-1$. For $n =1,2,\ldots$, 
\begin{equation} \label{eq:DSn=Sn}
   \partial S_n^{\a,\b} f = S_{n-1}^{\a+1,\b+1} (\partial f). 
\end{equation}
\end{lem}

\begin{proof}
It is well--known that the Jacobi polynomials are eigenfunctions of a second order differential operator
\cite[(4.21.1]{Sz}
$$       
  [w_{\a,\b}(t)]^{-1} \frac{d}{dt} \left[ (1-t^2) w_{\a,\b}(t)\right]  \frac{d}{dt} J_n^{\a,\b} = - \l_n J_n^{\a,\b},
$$
where $\l_n = n(n+\a+\b+1)$. Integrating by parts and applying this identity, we obtain by \eqref{eq:diffJ} that 
\begin{align*}
  \la f', J_n^{\a+1,\b+1}\ra_{\a+1,\b+1} & = \int_{-1}^1  f'(t)  \left[w_{\a+1,\b+1}(t) \frac{d}{dt}  J_{n+1}^{\a,\b}(t) \right] dt \\
    & =  - \int_{-1}^1  f(t) \frac{d}{dt} \left[w_{\a+1,\b+1}(t) \frac{d}{dt}  J_{n+1}^{\a,\b}(t) \right] dt \\
    & =  \l_{n+1}  \int_{-1}^1  f(t) J_{n+1}^{\a,\b}(t)w_{\a,\b}(t)  dt  = \l_{n+1} \la f, J_{n+1}^{\a,\b} \ra_{\a,\b}. 
\end{align*}
Setting $f =  J_{n+1}^{\a,\b}(t)$, it follows readily from the definition that 
\begin{equation} \label{eq:whDf=whf}
   \wh{\partial f}_n^{\a+1,\b+1} = \wh f_{n+1}^{\a,\b}, \qquad n =0,1,2,\ldots. 
\end{equation}
Consequently, by  \eqref{eq:diffJ} again, we see that 
\begin{align*}
  \partial S_n^{\a,\b} f (t)= \sum_{k=1}^n \wh f_k^{\a,\b} J_{k-1}^{\a+1,\b+1}(t) = \sum_{k=0}^{n-1} \wh{\partial f}_k^{\a+1,\b+1} 
      J_{k}^{\a+1,\b+1}(t) = S_{n-1}^{\a+1,\b+1} f(t).
\end{align*}
This completes the proof. 
\end{proof}


\begin{thm} 
Let $\a,\b > -1$ and $r \in \NN$. For $f \in L^2(w_{\a,\b})$ such that $f^{(r)} \in L^2(w_{\a+r,\b+r})$,  
\begin{equation} \label{eq:bestapp}
    E_n(f)_{L^2(w_{\a,\b})} \le n^{-r} E_{n-r}(f^{(r)})_{L^2(w_{\a+r,\b+r})}.
\end{equation}
\end{thm}

\begin{proof}
By the Parseval identity, \eqref{eq:whDf=whf} and the formula for $h_k^{\a,\b}$, 
\begin{align*}
  E_n(f)_{L^2(w_{\a,\b})}^2 =   \|f -S_n^{\a,\b} f \|_{L^2(w_{\a,\b})}^2 =&   \sum_{k=n+1}^\infty \left | \wh f_k^{\a,\b} \right|^2 h_k^{\a,\b} 
  = \sum_{k=n}^\infty \Big | \wh{\partial f}_k^{\a+1,\b+1} \Big |^2 h_{k+1}^{\a,\b} \\
= &  \sum_{k=n}^\infty \Big | \wh{\partial f}_k^{\a+1,\b+1} \Big |^2 \frac{h_k^{\a+1,\b+1}}{(k+1)(k+\a+\b+2)} \\
\le & \,  n^{-2}E_{n-1} (f')_{L^2(w_{\a+1,\b+1})}^2, 
\end{align*}
where we have used the Parseval identity again at the last step. Iterating this inequality proves the stated result. 
\end{proof}

The identity \eqref{eq:DSn=Sn} allows us to derive error estimates for simultaneous approximation by $S_n^{\a,\b} f$.
For comparison with our later results, we formulate the following corollary. 

\begin{cor}
Let $\a,\b > -1$ and $r, s \in \NN$. For $f \in L^2(w_{\a+r,\b+r})$ such that $f^r \in L^2(w_{\a,\b})$ and 
$n \ge r \ge s$, 
\begin{equation} \label{eq:simultan1}
    \left\| \phi^{k} \left(\partial^k f - \partial^k S_n^{\a,\b}\right) \right \|_{L^2(w_{\a,\b})} \le c n^{-r+k} E_{n-r} (f^{(r)})_{L^2(w_{\a,\b})}, \quad 
     0\le k \le s. 
\end{equation}
\end{cor}

\begin{proof}
By \eqref{eq:DSn=Sn}, the left hand side of \eqref{eq:simultan1} can be rewritten as $E_{n-k} (f^{(k)})_{L^2(w_{\a+k,\b+k})}$,
which is bounded by, by \eqref{eq:bestapp}, $n^{-r+k} E_{n-r} (f^{(r)})_{L^2(w_{\a+ r-k, \b + r-k})}$, in which we can remove 
$r-k$ since $w_{\a,\b}(t) \le 1$ if $\a,\b \ge 0$ and $r - k \ge 0$. 
\end{proof}

It is possible to remove $\phi^k$ in the left hand side of \eqref{eq:simultan1} with the penalty of a higher power of $n$ 
in the righthand side. This was first done in \cite{CQ}, see the proof in \cite{CHQZ}, for the Chebyshev and the Legendre 
cases with $r=1$ and later extended to the Gegenbauer weight in \cite{Guo}, but with $\|f^{(r)}\|_{L^2(w_{\a,\b})}$ in 
place of $E_{n-r} (f^{(r)})_{L^2(w_{\a,\b})}$ in \eqref{eq:simultan2} below, which is weaker than  \eqref{eq:simultan2} by 
\eqref{eq:Jackson2}. We give a complete proof for the Jacobi weight. 

\begin{thm} \label{thm:simultan2}
Let $\a,\b > -1$, $r =1,2,\ldots$, and $f \in W_2^r(w_{\a,\b})$. Then 
\begin{equation} \label{eq:simultan2}
    \left\| \partial^k f - \partial^k S_n^{\a,\b}\right \|_{L^2(w_{\a,\b})} 
         \le c_{\a,\b}\, n^{-r+2k-1/2} E_{n-r} (f^{(r)})_{L^2(w_{\a,\b})}, \quad 0 \le k \le r,
\end{equation}
where $c_{\a,\b}$ is proportional to $1/ \sqrt{\min\{\a,\b\}+1}$ when $k=1$. 
Moreover, the estimate \eqref{eq:simultan2} is sharp. 
\end{thm}

Comparing with \eqref{eq:thm1.2}, the order of $n$ in \eqref{eq:simultan2} is much weaker, which shows that
the least polynomials for $L^2(w_{\a,\b})$ is not suitable for simultaneous approximation. 

The proof of this theorem depends on two lemmas. The first one is an identity on the Jacobi polynomials. 

\begin{lem}
For $\a, \b > -1$ and $n \in \NN$, 
\begin{equation} \label{eq:Ja1b1=Jab}
J_n^{\a+1,\b+1}(t) = \sum_{j=0}^n C_{n,j}^{\a,\b} J_j^{\a,\b} (t), \qquad C_{n,j}^{\a,\b} := (-1)^{j+n} A_j^{\a,\b} B_n^{\a,\b} 
      + A_j^{\b,\a} B_n^{\b,\a}, 
\end{equation}
where 
$$
 A_j^{\a,\b} := \frac{(\a+\b+2)_{2j}}{(\a+1)_{j}} \quad\hbox{and}\quad B_n^{\a,\b}:= \frac{(\a+1)_{n+1}}{(\a+\b+2)_{2n+1}}.
$$
\end{lem}

\begin{proof}
The following relations on the Jacobi polynomials are stated in \cite{X16}, 
\begin{align}\label{eq:JacobiR1}
\begin{split}
 J_n^{\a,\b}(t) & =  J_n^{\a+1,\b}(t) - \tau_{n}^{\a,\b} J_{n-1}^{\a+1,\b}(t), \\ 
 J_n^{\a,\b}(t) & =  J_n^{\a,\b+1}(t) +  \tau_{n}^{\b,\a} J_{n-1}^{\a,\b+1}(t),  
\end{split}
\end{align}
where 
$\tau_{n}^{\a,\b}: = (n+\b)/((2n+\a+\b)(2n+\a+\b+1))$. Iterating these identities, it is easy to see that 
\begin{align*}
  J_n^{\a+1,\b} (t) = &\ \frac{(\b+1)_n}{(\a+\b+2)_{2n}} \sum_{k=0}^n \frac{(\a+\b+2)_{2k}}{(\b+1)_k} J_k^{\a,\b}, \\
  J_n^{\a,\b+1} (t) = &\ \frac{(\a+1)_n}{(\a+\b+2)_{2n}} \sum_{k=0}^n (-1)^{n-k} \frac{(\a+\b+2)_{2k}}{(\a+1)_k} J_k^{\a,\b}.   
\end{align*}
Together, these two identities imply that
\begin{align*}
J_n^{\a+1,\b+1} =  &\ \frac{(\b+2)_n}{(\a+\b+3)_{2n}} \sum_{k=0}^n \frac{(\a+\b+3)_{2k}}{(\b+2)_k} J_k^{\a,\b+1} \\
 =   &\ \frac{(\b+2)_n}{(\a+\b+3)_{2n}} \sum_{j=0}^n(-1)^j   
      \frac{(\a+\b+2)_{2j}}{(\a+1)_j} J_j^{\a,\b}\sum_{k=j}^n (-1)^k \frac{2k+\a+\b}{\a+\b+2}\frac{(\a+1)_k}{(\b+2)_k},
\end{align*}
where we have interchanged the order of summations. By induction on $n$, we can establish that 
$$
 \sum_{k=j}^n (-1)^k \frac{2k+\a+\b}{\a+\b+2}\frac{(\a+1)_k}{(\b+2)_k}  
 = (-1)^n \frac{\a+1}{\a+\b+2} \frac{(\a+2)_n}{(\b+2)_n} + (-1)^j \frac{\b+1}{\a+\b+2} \frac{(\a+1)_j}{(\b+1)_j}
$$
from which the stated result follows from a quick simplification. 
\end{proof}

\begin{rem}
The connection coefficients that appear when writing $J_n^{\a,\b}$ in terms of $J_n^{\g,\d}$ are non-negative if $\a = \g$ and
$\b > \d > -1$ or $\b = \d$ and $\a > \g > -1$, or $\a=\b > \g = \d > -1$. It is interesting to observe that the coefficients in \eqref{eq:Ja1b1=Jab} may not be all positive when $\a \ne \b$. For example, it is easy to see that the coefficient for $j =1$ and 
$n =4$ is negative if $\a > \b$. 
\end{rem}

Our main effort lies in establishing the identity \eqref{eq:main-lem} in the following lemma. 

\begin{lem} \label{lem:main-lem}
Let $\a,\b > -1$. If $ f\in W_2^1(w_{\a,\b})$, then 
\begin{align}\label{eq:main-lem} 
   S_{n-1}^{\a,\b}(f') - \partial S_n^{\a,\b}(f) = & \ \wh{\partial f}_n^{\a,\b} \sum_{j=0}^{n-1} 
      \left ( (-1)^{n+j} A_j^{\a,\b} B_n^{\a,\b}+A_j^{\b,\a}   B_n^{\b,\a} \right) J_j^{\a,\b}  \notag \\
       + & \wh{\partial f}_{n+1}^{\a,\b} \sum_{j=0}^{n-1}  \left ( (-1)^{n+j} A_j^{\a,\b} B_n^{\a,\b}D_n^{\a,b} -
         A_j^{\b,\a} B_n^{\b,\a}D_n^{\b,\a}\right)    J_j^{\a,\b},
\end{align}
where $A_j^{\a,\b}$ and $B_n^{\a,\b}$ are defined in \eqref{eq:Ja1b1=Jab} and 
$$
  D_j^{\a,\b} = \frac{(j+\b+1)}{(2j+\a+\b+2)(2j+\a+\b+3)}  .
$$
\end{lem}

\begin{proof}
First we assume that $f \in W_2^r(w_{\a,\b})$ for $r$ sufficiently large. Since $f' \in L^2 (w_{\a,\b})$, its Fourier 
orthogonal expansion is  
$$
    f' = \sum_{j=0}^\infty \wh {\partial f}_j^{\a,\b} J_j^{\a,\b}. 
$$
Moreover, since $L^2 (w_{\a,\b}) \subset L^2(w_{a+1,\b+1})$, we can also write, by \eqref{eq:Ja1b1=Jab}, that 
\begin{align*}
 f' =   \sum_{n=0}^\infty \wh {\partial f}_n^{\a+1,\b+1} J_n^{\a+1,\b+1} 
   &   = \sum_{n=0}^\infty \wh f_{n+1}^{\a,\b} \sum_{j=0}^n C_{n,j}^{\a,\b} J_j^{\a,\b} 
        = \sum_{j=0}^\infty  \bigg(\sum_{n=j}^\infty C_{n,j}^{\a,\b} \wh f_{n+1}^{\a,\b} \bigg) J_j^{\a,\b}.  
\end{align*}
Comparing the two expansions of $f'$, we conclude, by \eqref{eq:Ja1b1=Jab}, that 
\begin{equation}\label{eq:whf=F}
   \wh {\partial f}_j^{\a,\b}  =  \sum_{n=j}^\infty \wh f_{n+1}^{\a,\b} C_{n,j}^{\a,\b} =
      (-1)^j  A_j^{\a,\b} \Sigma_{1,j} +A_j^{\b,\a} \Sigma_{2,j},
\end{equation}
where 
$$
 \Sigma_{1,j} : =  \sum_{k=j}^\infty (-1)^k \wh f_{k+1}^{\a,\b} B_k^{\a,\b} \quad\hbox{and}\quad  \Sigma_{2,j}: = 
  \sum_{k=j}^\infty \wh f_{k+1}^{\a,\b} B_k^{\b,\a}. 
$$
The last two series are absolutely convergent, since $| \wh {\partial f}_{k+1}^{\a,\b}| (h_{k+1}^{\a,\b})^{\f12} \le E_k(f)_{\a,\b}$, 
which decays fast by \eqref{eq:bestapp}, and $B_k^{\a,\b}/(h_{k+1}^{\a,\b})^{\f12} \le c/ k^{2\b}$ and 
$B_k^{\a,\b}/(h_{k+1}^{\a,\b})^{\f12} \le c / k^{2\a}$. Since it is easy to check that 
$A_{j+1}^{\a,\b} B_j^{\b,\a} = A_{j+1}^{\b,\a} B_j^{\a,\b}$, we also have 
\begin{align} \label{eq:main-lem-2}
  \wh{\partial f}_{j+1}^{\a,\b} = \sum_{k=j+1}^\infty \wh f_{k+1}^{\a,\b} C_{j+1,k}^{\a,\b} = 
   (-1)^{j+1} A_{j+1}^{\a,\b} \Sigma_{1,j} + A_{j+1}^{\b,\a}  \Sigma_{2,j},
\end{align} 
where we have used $\Sigma_{1,j+1} = \Sigma_{1,j} - (-1)^j \wh f_{j+1}^{\a,\b} B_j^{\a,\b}$ and $\Sigma_{2, j+1} =  
\Sigma_{2,j} - \wh f_{j+1}^{\a,\b} B_j^{\b,\a}$. Solving \eqref{eq:whf=F} and \eqref{eq:main-lem-2} and simplifying, we
obtain
\begin{align} \label{eq:main-lem-3}
\begin{split}
\Sigma_{1,j} &\ = (-1)^j B_j^{\a,\b}  \Big(\wh{\partial f}_{j}^{\a,\b} -  D_j^{\a,\b} \wh{\partial f}_{j+1}^{\a,\b} \Big), \\ 
\Sigma_{2,j} & \ =   B_j^{\b,\a}  \Big( \wh{\partial f}_{j}^{\a,\b} +  D_j^{\b,\a} \wh{\partial f}_{j+1}^{\a,\b}\Big), 
\end{split}
\end{align}
Now, by \eqref{eq:diffJ} and \eqref{eq:Ja1b1=Jab},
$$
  \partial S_n^{\a,\b}(f) = \sum_{k=0}^{n-1} \wh f_{k+1}^{\a,\b} J_k^{\a+1,\b+1}
   = \sum_{j=0}^{n-1} \sum_{k=j}^{n-1} \wh f_{k+1}^{\a,\b} C_{k,j}^{\a,\b} J_j^{\a,\b},
$$
so that, by \eqref{eq:whf=F}, we conclude that 
\begin{align} \label{eq:main-lem-1}
  S_{n-1}^{\a,\b}(f') - \partial S_n^{\a,\b}(f) = &\ \sum_{j=0}^{n-1} \bigg(\wh f_j^{\a,\b} - 
         \sum_{k=j}^{n-1} \wh f_{k+1}^{\a,\b} C_{k,j}^{\a,\b}\bigg)J_j^{\a,\b} =\sum_{j=0}^{n-1}  
         \sum_{k=n}^\infty \wh f_{k+1}^{\a,\b} C_{k,j}^{\a,\b} J_j^{\a,\b}  \notag \\
         =&\ \sum_{j=0}^{n-1}  \left[ (-1)^j A_j^{\a,\b} \Sigma_{1,n} +  A_j^{\b,\a} \Sigma_{2,n}^{\b,\a}\right]J_j^{\a,\b}. 
\end{align}
Inserting the expressions for $\Sigma_{1,n}$ and $\Sigma_{2,n}$ in \eqref{eq:main-lem-2} completes the proof for 
smooth $f$. Since both sides of \eqref{eq:main-lem} is bounded for $f \in W_2^1(w_{\a,\b})$, as shown in the proof of
Theorem \ref{thm:simultan2}, the identity holds in $W_2^1(w_{\a,\b})$ by the usual density argument. 
\end{proof}

\bigskip\noindent
{\it Proof of Theorem \ref{thm:simultan2}.}
Assuming \eqref{eq:main-lem}, we proceed with the proof. First we consider the case $k=1$. Let $f \in W_2^r(w_{\a,\b})$.
By the triangle inequality, 
\begin{align*}
  \left\| \partial f - \partial S_n^{\a,\b} f \right \|_{L^2(w_{\a,\b})} \le 
    \left\| f' - S_{n-1}^{\a,\b} (f') \right \|_{L^2(w_{\a,\b})}+ \left \| S_{n-1}^{\a,\b}(f') - \partial S_n^{\a,\b}(f) \right \|_{L^2(w_{\a,\b})}.
\end{align*}
The first term in the right hand side is bounded by $E_{n-1}(f')_{\a,\b}$, which is small than the desired bound. We now 
bound the second term. By \eqref{eq:main-lem} 
\begin{align*}
   \|S_{n-1}^{\a,\b}(f') - \partial S_n^{\a,\b}(f)\|_{L^2(w_{\a,\b})}^2 = & \ \Big| \wh{\partial f}_n^{\a,\b}\Big|^2
       \sum_{j=0}^{n-1}  \left |(-1)^{n+j} A_j^{\a,\b} B_n^{\a,\b}+A_j^{\b,\a}   B_n^{\b,\a} \right| ^2 h_j^{\a,\b}  \notag \\
        + \Big |   \wh{\partial f}_{n+1}^{\a,\b}\Big|^2 & \sum_{j=0}^{n-1}  \left | (-1)^{n+j} A_j^{\a,\b} B_n^{\a,\b}D_n^{\a,\b} 
           - A_j^{\b,\a} B_n^{\b,\a}D_n^{\b,\a}\right |^2 h_j^{\a,\b}.
\end{align*}
By the expression of $A_j^{\a,\b}$ and $B_n^{\a,\b}$ in \eqref{eq:Ja1b1=Jab}, it is not difficult to verify that 
\begin{align*}
 \frac{\left|B_n^{\a,\b}\right|^2}{h_n^{\a,\b} } \sum_{j=0}^{n-1} \Big | A_j^{\a,\b} \Big|^2 h_j^{\a,\b} = & 
   \frac{n(n+\a) (n+\a+1)^2}{(\b+1)(2n+\a+\b+1)(2n+\a+\b+2)^2} \sim \frac{n}{\b+1} \\
\frac{\left|B_n^{\b,\a}\right|^2}{h_n^{\a,\b} } \sum_{j=0}^{n-1} \Big | A_j^{\b,\a} \Big|^2 h_j^{\a,\b} = & 
   \frac{n(n+\b) (n+\b+1)^2}{(\a+1)(2n+\a+\b+1)(2n+\a+\b+2)^2} \sim \frac{n}{\a+1}, 
\end{align*}
and $h_n^{\a,\b} |D_n^{\a,\b}|^2/h_{n+1}^{\a,\b} \sim 1$. Consequently, we deduce that 
\begin{align*}
 \|S_{n-1}^{\a,\b}(f') - \partial S_n^{\a,\b}(f)\|_{L^2(w_{\a,\b})}^2 &\  
     \le c_{\a,\b}\ n \left( \Big| \wh{\partial f}_n^{\a,\b}\Big|^2 h_n^{\a,\b} +  
 \Big| \wh{\partial f}_{n+1}^{\a,\b}\Big|^2 h_{n+1}^{\a,\b} \right) \\
  &\ \le c_{\a,\b}\ n \left[ E_{n-1}(f')_{\a,\b} \right]^2,
\end{align*}
where the last step follows from the Parseval identity. This proves \eqref{eq:simultan2} for $k=1$ and $r =1$, which implies
the case $k=1$ and $r \ge 1$ by \eqref{eq:bestapp}. 

The case $k > 1$ follows inductively. Our main effort lies in proving the inequality 
\begin{equation} \label{eq:simu-1}
  \left \| \partial^m \left(S_{n-1}^{\a,\b}(f')  - \partial  S_n^{\a,\b} f \right)\right \|_{L^2(w_{\a,\b})} \le c\, n^{2m+1} E_{n-1}(f')_{\a,\b} 
\end{equation}
for $m=1,2,\ldots$. Using \eqref{eq:main-lem} and $\partial^m J_j^{\a,\b} = J_{j-m}^{\a+m,\b+m}$, we see that the
main ingredient is the estimate the sum 
\begin{align*}
 \bigg \| \sum_{j=m}^{n-1}   A_j^{\a,\b} \left|J_{j-m}^{\a+m,\a+m} \right| \bigg \|_{L^2(w_{\a,\b})}^2  
  &   \le n \sum_{j=m}^{n-1} \Big| A_j^{\a,\b}\Big|^2 h_{j-m}^{\a+m,\b+m}  \\
  &   \le c\, n \sum_{j=m}^{n-1} \Big| A_j^{\a,\b}\Big|^2  h_{j-m}^{\a,\b}\, j^{2m-1}  \le c\, n^{4m+2\b+2}.
\end{align*}
where the first inequality follows from the Cauchy-Schwartz inequality and the second one follows from 
\begin{equation}\label{eq:simu-2}
h_n^{\a+m,\b+m} / h_n^{\a,\b} \sim n^{2m-1},
\end{equation}
which can be easily verified by \eqref{eq:h_n} and the asymptotic of the Gamma function, and the third one follows 
from a straightforward estimate. Consequently, it follows readily that 
$$
  \frac{\left| B_n^{\a,\b} \right|^2}{h_n^{\a,\b}} 
    \bigg \| \sum_{j=m}^{n-1} A_j^{\a,\b} \left| J_{j-m}^{\a+m,\a+m} \right| \bigg \|_{L^2(w_{\a,\b})}^2  \le c \, n^{2m+1}
$$
and the similar estimate holds when $\a$ and $\b$ are exchanged in $A_j^{\a,\b}$ and $B_n^{\a,\b}$. These estimates allow us
to estimate, by \eqref{eq:main-lem}, that 
\begin{align*}
  \left \| \partial^m \left(S_{n-1}^{\a,\b}(f')  - \partial  S_n^{\a,\b} f \right)\right \|_{L^2(w_{\a,\b})} 
    \le c\, n^{2m+1}  \left( \Big| \wh{\partial f}_n^{\a,\b}\Big|^2 h_n^{\a,\b} +  
        \Big| \wh{\partial f}_{n+1}^{\a,\b}\Big|^2 h_{n+1}^{\a,\b} \right),
\end{align*}
from which \eqref{eq:simu-1} follows readily. 

Assume now \eqref{eq:simultan2} has been established for a fixed $k$, we prove that it also holds for $k+1$. By the 
triangle inequality, 
\begin{align*}
 & \left\| \partial^{k+1} f - \partial^{k+1} S_n^{\a,\b} f \right \|_{L^2(w_{\a,\b})} \\
&  \qquad\quad   \le 
    \left\| \partial^k f'- \partial^k  S_{n-1}^{\a,\b} (f') \right \|_{L^2(w_{\a,\b})}+ \left \|\partial^k 
           \Big[S_{n-1}^{\a,\b}(f') - \partial S_n^{\a,\b}(f)\Big]\right \|_{L^2(w_{\a,\b})}.
\end{align*}
The second term is the right hand side can be bounded, by applying \eqref{eq:simu-1} with $m=k$ and \eqref{eq:bestapp}, by 
$c\, n^{2k+1/2} E_{n-1}(f')_{\a,\b} \le c\, n^{2k+1/2} n^{-r+1} E_{n-r}(f^{(r)})_{\a,\b}$, in which the power of $n$ can be written
as $-r+2(k+1)-1/2$, which agrees with that in \eqref{eq:simultan2} for $k+1$, whereas the first term in the right hand side can be 
bounded, by induction hypothesis with $r$ replaced by $r-1$, by a bound that is less than the above bound. This completes the 
proof of \eqref{eq:simultan2} for $k=1$ and the proof.  

To show that the order is sharp, we consider $g(t) = \wh J_{n+1}^{\a-k,\b-k}(t)$, which is well defined for $n$ large even if 
$\a <0$ or $\b < 0$. Then $g^{(k)}(t) = \wh J_{n+1-k}^{\a,\b}(t)$. Since the orthogonal expansion of $\wh J_{n+1 -k}^{\a,\b}$ is
itself, $E_{n-k}(g^{(k)})_{\a,\b} = \|\wh J_{n+1-k}^{\a,\b}\|_{\a,\b}$. Furthermore, by \eqref{eq:JacobiR1}, 
$J_{n+1-k}^{\a,\b} - J_{n+1-k}^{\a+k,\b+k}$ is a polynomial of degree $n-k$, so that 
$$
  \partial^k S_n^{\a,\b} g (t) = S_{n-k}^{\a+k,\b+k}(g^{(k)}) =  S_{n-k}^{\a+k,\b+k}(\wh J_{n+1-k}^{\a,\b}) = 
      \wh J_{n+1-k}^{\a,\b} - \wh J_{n+1-k}^{\a+k,\b+k}.
$$  
Consequently, $\partial^k g -   \partial^k S_n^{\a,\b} g = \wh J_{n+1-k}^{\a+k,\b+k}$. It then follows from \eqref{eq:simu-2}
that \eqref{eq:simultan2} is sharp for $k=r$. 
\qed

\section{Sobolev orthogonal polynomials and orthogonal expansions} \label{sect3}
\setcounter{equation}{0}

As mentioned in the introduction, for approximation in the Sobolev space $W_p^s(w_{\a,\b})$, we need to work with
the Jacobi polynomials with parameters $\a,\b$ being negative integers. Setting $\a,\b$ as negative integers in 
\eqref{eq:JacobiP} leads to a reduction of polynomial degrees, which causes problems when one considers 
orthogonal expansions. There have been several ways of remedying the definition of the Jacobi polynomials in 
the literature; see, for example, \cite{AAR02, GW, LX, APP98} and the references therein. Motivating by the study 
in \cite{X16}, which will be explained in the end of this section, we give another definition that can be regarded as 
either avoiding delicate extensions of the Jacobi polynomials to negative integers or as an alternative definition 
that holds for all negative indices.  

\begin{defn}
Let $\a,\b > -1$ and $s \in \NN$. For $\t \in [-1,1]$ and $n \in \NN_0$, define
\begin{align} \label{eq:CJ}
\begin{split}
   \CJ_n^{\a-s,\b-s}(x) = \CJ_{n,\t}^{\a-s,\b-s}(x):= \begin{cases}  \dfrac{(x-\t)^n}{n!}, & 0 \le n \le s-1, \vspace{.05in} \\  
        \displaystyle{\int_\t^x \frac{(x-t)^{s-1}}{(s-1)!} J_{n-s}^{\a,\b}(t) dt}, & n \ge s.
   \end{cases}
\end{split}
\end{align}
\end{defn}

It is evident that $\CJ_n^{\a-s,\b-s}$ is a polynomial of degree $n$. Furthermore, these polynomials evidently 
satisfy the following properties: 
\begin{align} 
 \partial^s \CJ_n^{\a-s,\b-s}(x) &\ = J_{n-s}^{\a,\b}(x), \qquad  n \ge s; \label{eq:CJ1a}\\
 \partial^k \CJ_n^{\a-s,\b-s}(\t) &\ = \begin{cases} \delta_{k,n}, & n \le s-1, \\ 
         0, &  n \ge s,
         \end{cases} \qquad 0\le k \le s-1, \label{eq:CJ1b}
\end{align}
where $\partial^k$ denotes the $k$-th derivative. Comparing with \eqref{eq:diffJ}, the identity \eqref{eq:CJ1a} 
suggests that these polynomials can be regarded as an extension of the Jacobi polynomials with negative parameters when 
$\a- s \le -1$ and/or $\b -s \le -1$. If $\a-s > -1$ and $\b-s > -1$, then both $J_n^{\a-s,\b-s}$ and $\CJ_n^{\a-s,\b-s}$ 
satisfy \eqref{eq:CJ1a}, but $J_n^{\a-s,\b-s}$ does not satisfy \eqref{eq:CJ1b}. These polynomials are orthogonal
with respect to the inner product
$$ 
\la f, g\ra_{\a,\b}^{-s} : = \int_{-1}^1 f^{(s)}(t) g^{(s)}(t) w_{\a,\b}(t) dt + \sum_{k=0}^{s-1} \l_k f^{(k)}(\t)g^{(k)}(\t), 
$$
where $\l_k$ are positive constants. 

\begin{thm} \label{thm:ortho-s}
For $\a,\b > -1$ and $s \in \NN$. The polynomial $\CJ_{n,\t}^{\a-s,\b-s}$ is orthogonal with respect to the inner product
$\la \cdot, \cdot \ra_{\a,\b}^{-s}$ and its norm square, $\mathfrak{h}_{n}^{\a-s,\b-s} :=  \la  \CJ_n^{\a-s,\b-s}, 
\CJ_n^{\a-s,\b-s}\ra_{\a,\b}^{-s}$, satisfies
$$
   \mathfrak{h}_{n}^{\a-s,\b-s} = \l_n, \quad 0 \le n \le s-1, \quad \hbox{and} \quad \mathfrak{h}_{n}^{\a-s,\b-s} = h_{n-s}^{\a,\b}, \quad n \ge s.
$$
\end{thm}

\begin{proof}
Let $m  \le n$. We consider the orthogonality of $\CJ_n^{\a-s,\b-s}$ and $\CJ_m^{\a-s,\b-s}$. If $n \le s-1$ then, 
by \eqref{eq:CJ1a} and \eqref{eq:CJ1b},
$$
 \la \CJ_n^{\a-s,\b-s},  \CJ_m^{\a-s,\b-s}\ra_{\a,\b}^{-s} = 
       \sum_{k=0}^{s-1} \l_k \partial^k \CJ_n^{\a-s,\b-s}(\t) \partial^k \CJ_m^{\a-s,\b-s}(\t) = \l_n \delta_{n,m}.
$$
Whereas if $n \ge s$, then, by \eqref{eq:CJ1a} and \eqref{eq:CJ1b},
$$
 \la \CJ_n^{\a-s,\b-s}, \CJ_m^{\a-s,\b-s}\ra_{\a,\b}^{-s} =  \int_{-1}^1 J_{n-s}^{\a,\b}(x) J_{m-s}^{\a,\b}(x) w_{\a,\b}(x) dx = 
   \delta_{n,m} h_{n-s}^{\a,\b} 
$$
by the orthogonality of the Jacobi polynomials. 
\end{proof}

As we mentioned before, the inner product $\la \cdot,\cdot\ra_{\a,\b}$ and its associated orthogonal polynomials 
have been studied in the literature (see, \cite{MX} and its references). Instead of starting with an extension of the 
Jacobi polynomials to parameters being negative integers and constructing orthogonal polynomials accordingly, 
our construction is more direct with a strikingly, in comparison, simple proof and works for all real parameters. 

For $f \in \CW_p^s(w_{\a,\b})$, we can study the Fourier orthogonal expansion of $f$ with respect to the 
orthogonal system $ \CJ_{n,\t}^{\a-s,\b-s}$, 
$$
  f = \sum_{n=0}^\infty \wh \ff_n^{\a-s,\b-s} \CJ_{n,\t}^{\a-s,\b-s} \quad \hbox{with}\quad
       \wh \ff_n^{\a-s,\b-s} = \wh \ff_{n,\t}^{\a-s,\b-s} := \frac{\la f, \CJ_{n,\t}^{\a-s,\b-s}\ra_{\a,\b}^{-s}} { \mathfrak{h}_n^{\a-s,\b-s}}.
$$
The $n$-th partial sum of this expansion is defined by 
$$
  \CS_n^{\a-s,\b-s} f =\CS_{n,\t}^{\a-s,\b-s} f := \sum_{k=0}^n \wh \ff_k^{\a-s,\b-s} \CJ_{k,\t}^{\a-s,\b-s}.
$$
For $s =0$, the operator $\CS_n^{\a,\b} f = S_n^{\a,\b} f$ is the partial sum of the usual Jacobi expansion. 
This operator satisfies several simple properties and can be written, in particular, in terms of the partial sum 
$S_{n-s}^{\a,\b} f$, where we define $S_n^{\a,\b} f =0$ if $n < 0$. 

\begin{lem}
Let $\a,\b > -1$ and $s \in \NN$. For $f \in W_p^s(w_{\a,\b})$,
\begin{enumerate}
\item $\wh \ff_n^{\a-s,\b-s} = f^{(n)}(\t)$ if $0 \le n \le s-1$, and $\wh \ff_n^{\a-s,\b-s} = \wh {\partial^s f}_{n-s}^{\a,\b}$ if $n \ge s$;
\item $\partial^s \CS_n^{\a-s,\b-s} f = S_{n-s}^{\a,\b} f^{(s)}$ if $n \ge s$; 
\item For $n = 0,1, \ldots$,
$$
    \CS_n^{\a-s,\b-s} f(x) = \sum_{k=0}^{\min\{n,s-1\}}f^{(k)}(\t) \frac{(x-\t)^k}{k!} + \int_\t^x \frac{(x-t)^{s-1}}{(s-1)!} S_{n-s}^{\a,\b} f^{(s)}(t) dt.
$$
\end{enumerate}
\end{lem}

\begin{proof}
As in the proof of the Theorem \ref{thm:ortho-s}, it is easy to see that, if $n \le s-1$, then 
$\la f,  \CJ_n^{\a-s,\b-s}\ra_{\a,\b}^{-s}  = \l_n f^{(n)}(\t)$ and  $\hf_n^{\a-s,\b-s}  = \l_n$, whereas if $n \ge s$ then 
$\la f^{(s)}, J_{n-s}^{\a,\b}\ra_{\a,\b}$ and $\hf_n^{\a-s,\b-s} = h_{n-s}^{\a,\b}$ for $n \ge s$, from which (1) for $n \ge s$ 
follows readily. The case for $n \le s-1$ follows similarly. By the definition of the partial sum operator, we then obtain,
for $n \ge s$, 
\begin{align*}
 \CS_n^{\a-s,\b-s} f(x)  = &\ \sum_{k= 0}^{s-1} f^{(k)}(\t) \CJ_{k}^{\a-s,\b-s}(x) + 
        \sum_{k=s}^n \wh {\partial^s f}_k^{\a,\b} \CJ_{k}^{\a-s,\b-s}(x)  \\
        = &\ \sum_{k= 0}^{s-1} f^{(k)}(\t) \frac{(x-\t)^k}{k!} + 
        \sum_{k=s}^n \wh {\partial^s f}_k^{\a,\b} \int_\t^x \frac{(x-t)^{s-1}} {(s-1)!} J_{k-s}^{\a,\b}(t)dt,
\end{align*}
which is, after changing the order of the sum and the integral, exactly the right hand side of (3). 
Finally, taking $s$-th derivative of the identity in (3) proves (2). 
\end{proof}

In \cite{LX}, we extended the Jacobi polynomials to allow the parameters to be negative integers. The extended polynomial,
again denoted by $J_n^{\a,\b}$, satisfies \eqref{eq:diffJ} for all $\a,\b \in \RR$, $n \in \NN$, and its leading coefficient is 
$x^n/n!$, that is, $J_n^{\a,\b}(x)= x^n/n! + \ldots$. Based on this extension of the Jacobi polynomials, we proved in 
\cite{X16} that the
polynomials
$$
  \wh J_n^{-s,\b-s}(t) = \begin{cases} J_n^{-s,\b-s} (t), &  n \ge s, \\
      \displaystyle{ J_n^{-s,\b-s} (t) - \sum_{k=0}^{s-1} J_{n-k}^{-s+k,\b-s+k} (1) \frac{(t-1)^k}{k!}}, & n < s
      \end{cases}
$$ 
are orthogonal with respect to the inner product $\la \cdot,\cdot\ra_{0,\b}^{-s}$ with $\t =1$. Below we prove that
$\wh J_n^{-s,\b-s}(x) \equiv \CJ_n^{-s,\b-s}(x)$ for $n \ge 0$. Indeed, if $n<s$, we use \eqref{eq:diffJ} and the leading 
coefficient of $J_n^{\a,\b}$ to conclude, by the Taylor expansion, that
$$
 \wh J_n^{-s,\b-s}(x) = J_n^{-s,\b-s}(x) - \sum_{k =0}^{n-1} \partial^k J_n^{-s,\b-s}(1) \frac{(x-1)^k}{k!} 
    =  \frac{(x-1)^n}{n!} = \CJ_n^{-s,\b-s}(x), 
$$
whereas if n $\ge s$, the equivalence follows from the identity \cite[(2.9)]{X16}
\begin{align*}
J_n^{-s,\b-s}(x)  = \frac{(-1)^s (n-s)!}{n!} (1-x)^s J_{n-s}^{0,\b}(x), \quad n \ge s,
\end{align*}
and the integral form of the Taylor reminder formula. In fact, it is this equivalence that motivated our definition 
of $\CJ_n^{\a-s,\b-s}$ for all $\a, \b \in \RR$.

\section{Polynomial approximation in Sobolev spaces} \label{sect4}
\setcounter{equation}{0}

We shall show that $\CS_{n,\t}^{\a-s,\b-s}f $ approximates $f$ with the least error, up to a multiple constant, in the 
$W_2^s(w_{\a,\b})$ norm. We also define a near best approximation polynomial that will play the role of $\CS_{n,\t}^{\a,\b} f$ 
when $p\ne 2$. Fix $\t$, we define $\CV_n^{\a-s,\b-s} =\CV_{n,\t}^{\a-s,\b-s}$ for $s =1,2,\ldots$ by 
\begin{equation}\label{eq:approxSP-b}
   \CV_{n,\t}^{\a-s,\b-s}f (x) :=  \sum_{k=0}^{s-1} f^{(k)}(\t) \frac{(x-\t)^k}{k!}+ \int_\t^x \frac{(x-t)^{s-1}}{(s-1)!} 
      V_{n}^{\a,\b} f^{(s)}(t) dt.
\end{equation}
It is easy to see that $\CV_{n,\t}^{\a-s,\b-s}f$ is a polynomial of degree $2n+s$ and it preserves polynomials of 
degree $\le n$. Furthermore, it satisfies
\begin{equation}\label{eq:approxSP-c}
    \partial^s \CV_{n,\t}^{\a-s,\b-s}f (x) = V_{n}^{\a,\b} f^{(s)}(x). 
\end{equation}

To study the approximation property of $\CS_{n,\t}^{\a,\b}f$ or $\CV_{n,\t}^{\a-s,\b-s}f$, we will need the 
weighted Hardy inequality with respect to the Jacobi weight: 

\begin{lem}\label{lem:Hardy}
For $\a,\b > -1$ and $f \in L^p(w_{\a,\b})$, $1 < p < \infty$,  
$$
   \left(\int_{-1}^1  \left (\int_{-1}^x |f(t)| dt \right )^p w_{\a,\b}(x) dx\right)^{1/p}
          \le c \left( \int_{-1}^1 |f(x)|^p w_{\a,\b}(x) dx \right)^{1/p}
$$
if and only if $\b < p-1$. 
\end{lem}
 
\begin{proof}
With $w_{\a,\b}$ replaced by a general weight function $w$, it is known (see, for example, \cite{OK})
that the inequality holds if and only if 
$$
  \sup_{-1<x<1} \left( \int_{x}^1 w(t) dt \right)^{1/p} \left(\int_{-1}^x w(t)^{1-q} dt\right)^{1/q} < \infty, 
$$
where $q = p/(p-1)$. If $-1 < x \le 0$, then the first integral is finite and the second integral is finite if and only if, since 
$1-t \sim 1$, $\b (1-q) > -1$, or $\b < 1/(q-1) = p-1$. If $0 \le x < 1$, then the first integral is bounded by $c (1-x)^{(\a+1)/p}$, 
whereas the second integral is bounded by, since $1+t \sim 1$, $(1-x)^{(\a(1-q)+1)/q}$, it is easy to verify that 
their product is finite. 
\end{proof} 
 
\begin{thm}
Let $\a, \b > -1$ and $s \in \NN$. For $f\in W_p^s(w_{\a,\b})$ if $1 \le p < \infty$, or $f \in C^s[-1,1]$ if $p = \infty$, the
estimate 
\begin{equation} \label{eq:approxSP}
   \| f - \CV_{n,\t}^{\a-s,\b-s} f \|_{W_p^s(w_{\a,\b})} \le c\, E_{n}(f^{(s)})_{L^p(w_{\a,\b})}
\end{equation}
holds under either one of the following assumptions:  
\begin{enumerate}
\item $\t \in (-1,1)$; 
\item $\t =-1$, $ \b < p-1$ if $1< p \le \infty$ and $\b \le 0$ if $p=1$;
\item $\t=1$, $\a < p-1$ if $1< p \le \infty$ and $\a \le 0$ if $p=1$.
\end{enumerate}
\end{thm}

\begin{proof}
By the definition of the $\|\cdot\|_{W_p^s(w_{\a,\b})}$ norm, we need to show that
\begin{equation} \label{eq:approxSP-a}
 \| \partial^k f -\partial^k \CV_{n,\t}^{\a-s,-s} f \|_{L^p(w_{\a,\b})} \le c\,  E_n(f^{(s)})_{L^p(w_{\a,\b})}, \quad 0 \le k \le s.
\end{equation}
Since $\partial^s f - \partial^s \CV_{n,\t}^{\a-s,\b-s} f = f^{(s)} - V_{n}^{\a,\b}f^{(s)}$, the estimate \eqref{eq:approxSP-a} 
when $k =s$ follows immediately. By the Taylor reminder formula, we can write 
$$
  f(x) = \sum_{k=0}^{s-1} f^{(k)}(\t) \frac{(x-\t)^k}{k!} + \int_\t^x \frac{(x-t)^{s-1}}{(s-1)!} f^{(s)}(t) dt,
$$
so that 
$$
  f(x) - \CV_{n,\t}^{\a-s,\b-s} f (x) =  \int_\t^x \frac{(t-\t)^{s-1}}{(s-1)!} \left[ f^{(s)}(t) - V_{n}^{\a,\b}  f^{(s)}(t) \right] dt
$$
Taking $k$-th derivatives, $1 \le k \le s-1$, only reduces the power of $(t-\t)^{s-1}$ by $k$. Since we need to consider
$k =0,1,\ldots, s-1$, we can ignore the factor $(t-\t)^{s-k-1}$. For $p = \infty$, the estimate \eqref{eq:approxSP-a} follows
trtivially from the above identity. 

For $p > 1$, applying the H\"older inequality, we obtain
\begin{align*}
  &\left \| f- \CV_{n,\t}^{\a-s, \b-s} f \right\|_{L^p(w_{\a,\b})}^p  \le
      \int_{-1}^1 \left | \int_\t^x |f^{(s)}(t) - V_{n}^{\a,\b} f^{(s)}(t)|dt \right|^p  w_{\a,\b}(x) dx\\
   & \qquad \qquad \le \int_{-1}^1 \left| \int_\t^x | f(t)- V_{n}^{\a, \b} f (t)|^p w_{\a,\b}(t) dt\right| \cdot 
       \left | \int_\t^x w_{\a,\b}(t)^{-q/p} dt \right|^{p/q}  w_{\a,\b}(x) dx \\
   & \qquad \qquad
      \le  \left \|f^{(s)} - V_{n}^{\a,\b} f^{(s)}\right\|_{L^p(w_{\a,\b})}^p  \int_{-1}^1 \left |\int_\t^x w_{\a,\b}(t)^{-q/p} dt \right|^{p/q}
          w_{\a,\b}(x) dx. 
 \end{align*}
If $\t \in (-1,1)$, it is easy to see, since $w_{\a,\b}$ has no singularity at $\t$, that the last integral is bounded by a constant 
for all $p > 1$. If $ \t = -1$, the inequality follows from Lemma \ref{lem:Hardy} and the case $\t = 1$ follows likewise
from the dual Hardy inequality (\cite{OK}). A direct proof can also be easily carried out by dividing the integral over 
$[-1,1]$ into two integrals over $[-1,0]$ and $[0,1]$, respectively, so that 
\begin{align*}
 \int_{-1}^1 \left |\int_{-1}^x w_{\a,\b}(t)^{-q/p} dt \right|^{p/q} w_{\a,\b}(x) dx 
    \le &\ c \int_{-1}^0 \left |\int_{-1}^x (1+t)^{-\b q/p} dt \right|^{p/q} w_{\a,\b}(x) dx \\
       & + \int_{0}^1 \left |\int_{-1}^x w_{\a,\b}(t)^{-q/p} dt \right|^{p/q} w_{\a,\b}(x) dx; 
\end{align*}
the first term in the right hand side is bounded if $ - \b q/p > -1$ or $\b < p/q = p-1$, whereas the second term can be seen
to be bounded, after splitting the inner integral as a sum of two, one over $[-1,0]$ and the other over $[0,x]$. The case $\t =1$
works similarly. 

For $ p =1$, we divide the integral into two terms and exchanging the orders of the integrals in each term,
\begin{align*}
   \left \| f- \CV_{n,\t}^{\a-s, \b-s} f \right\|_{L^1(w_{\a,\b})}  \le &   \left(  \int_{-1}^\t  \int_x^\t  +  \int_{\t}^1  \int_\t^x  \right)
        \left|f^{(s)}(t) - V_{n}^{\a,\b} f^{(s)}(t)\right| dt 
        w_{\a,\b}(x) dx \\
      =  & \int_{-1}^\t  \left|f^{(s)}(t) - V_{n}^{\a,\b} f^{(s)}(t)\right| \int_{-1}^t w_{\a,\b}(x)dx dt  \\ 
   & +   \int_{\t}^1  \left|f^{(s)}(t) - V_{n}^{\a,\b} f^{(s)}(t)\right|  \int_t^1  w_{\a,\b}(x) dx dt.   
\end{align*}
For $ \t \in (-1,1)$, $w_{\a,\b}(x) \sim (1-x)^\a$ in the first term in the right hand side, which implies that 
$\int_{-1}^t w_{\a,\b}(x)dx \le c w_{\a,\b}(t)$, whereas $w_{\a,\b}(x) \sim (1+x)^\b$ in the second term, which implies that
$\int_t^1  w_{\a,\b}(x) dx \le c w_{\a,\b}(t)$. This proves the case for $\t \in (-1,1)$. For $\t = -1$, there is only the second term,
for which we use $(1+x)^\b \le (1+t)^\b$ for $\b \le 0$ to conclude that $\int_t^1  w_{\a,\b}(x) dx \le c w_{\a,\b}(t)$ for $\b \le 0$.
The case $\t= 1$ is proved similarly. The proof is completed. 
\end{proof}
 
Evidently, Theorem \ref{thm:1.1} follows from this theorem with any $\t \in (0,1)$. In the case of $p=2$, we 
can replace $\CV_{n,\t}^{\a-s,\b-s}f $ by the partial sum operator $\CS_n^{\a-s,\b-s} f$.  
 
\begin{cor}
Let $\a, \b > -1$, $s \in \NN$. For $f\in W_2^s(w_{\a,\b})$, the estimate 
\begin{align} \label{eq:approxSP2}
   \| f - \CS_{n,\t}^{\a-s,\b-s} f \|_{W_2^s(w_{\a,\b})}  \le  c \, E_{n-s}(f^{(s)})_{L^2(w_{\a,\b})} \end{align}
holds if (a) $\t \in (-1,1)$, or (b) $\t =-1$ and $ \b < 1$, or (c) $\t=1$ and $\a < 1$. 
\end{cor}

In particular, if $f \in W_2^r(w_{\a,\b})$ for $r \ge s$, then we have
$$
    \| f - \CS_{n}^{\a-s,\b-s} f \|_{W_2^s(w_{\a,\b})} \le c \,n^{-r+s} E_{n-r}(f^{(r)})_{L^2(w_{\a+r-s,\b+r-s})}
$$
and the righthand is also bounded by $c\, n^{-r+s} \|\phi^{r-s} f^{(r)}\|_{L^p(w_{\a,\b})}$ by Theorem \ref{thm:Jackson}. 

\section{Simultaneous Approximation}
\setcounter{equation}{0}

We now consider the simultaneous approximation by polynomials. Our goal is to establish the estimate of the type 
$$
  \| \partial^k f - \partial^k \CV_{n,\t}^{\a-s,\b-s} f \|_{L^p(w_{\a,\b})} \le c \, n^{-s+k} E_n(f^{(s)})_{L^p(w_{\a,\b})},
  \quad 0 \le k \le s. 
$$
 Our result in the previous section shows that this estimate holds for $k=s$. We apply a technique called 
Aubin-Nitsche duality (see \cite[p. 321]{CHQZ} and \cite{CQ}) to handle the case $0 \le k \le s-1$. For this, we need to work
with either $\t = -1$ or $\t =1$. We mainly work with $\t =-1$, since the other case is similar. 

Let $g$ be a measurable function on $[-1,1]$. For $ 0\le k \le s-1$, we define a function 
\begin{equation}\label{eq:ug}
  u_{g,k}(x):= \int_{-1}^x \frac{(x-t)^{s-1}}{(s-1)!} \frac{1}{w_{\a,\b}(t)}  \int_t^1 \frac{(y-t)^{s-k-1}}{(s-k-1)!} g(y) w_{\a,\b}(y) dy dt. 
\end{equation}
It is straightforward to verify that $u_{g,k}$ is the solution of the following boundary value problem of a $(2s-k)$-th 
order differential equation:
\begin{equation} \label{eq:diffq}
\begin{split}
 & (-1)^{s-k} [w_{\a,\b}(x)]^{-1}  \frac{d^{s-k}}{dx^{s-k}} \left(w_{\a,\b}(x) u^{(s)}(x) \right)  = g (x) , \quad -1 \le x \le 1, \\
 & \qquad u^{(j)} (-1) =0, \quad 0\le j \le s-1, \\
 &  \qquad \lim_{x\to 1}   \frac{d^j}{dx^j} \left(w_{\a,\b}(x) u^{(s)}(x)\right)  =0, \quad 0\le j \le s-k-1. 
\end{split}
\end{equation} 


\begin{lem} \label{lem:int-g}
If $v \in W_p^s$ satisfies $v^{(j)}(-1) =0$ for $0 \le j \le s-k-1$, then 
$$
  \int_{-1}^1 u_{g,k}^{(s)}(x) v^{(s-k)}(x) w_{\a,\b}(x) dx = \int_{-1}^1 g(x) v(x) w_{\a,\b}(x) dx. 
$$
\end{lem}

\begin{proof}
By the differential equation in \eqref{eq:diffq}, 
\begin{align*}
 \int_{-1}^1 g(x) v(x) w_{\a,\b}(x) dx = (-1)^{s-k} \int_{-1}^1 \frac{d^{s-k}}{dx^{s-k}} \left(w_{\a,\b}(x) u_{g,k}^{(s)}(x) \right) v(x) dx.
\end{align*}
With the help of the boundary conditions of $u_{g,k}^{(s)}$ at $x=1$ and $v$ at $-1$, successive integration by parts proves 
the stated identity. 
\end{proof}

Throughout the rest of the paper, we let $p$ and $q$ be related by $\f1{p} + \frac1{q} =1$. 

\begin{lem}
Let $\a > -1$ and $g \in L^q(w_{\a,\b})$, $ 1 \le q \le \infty$. Then, for $1 \le q \le \infty$, 
\begin{equation}\label{eq:ug<g}
 \|\phi^{s-k} u_{g,k}^{(2s-k)}\|_{L^q(w_{\a,\b})} \le c \|g\|_{L^q(w_{\a,\b})}, \quad 0 \le k \le s-1, 
\end{equation}
provided (a) $\b =0$ or (b) $s=1$, $\b < p/2-1$ if $1 \le q < \infty$ and $\b \le -1/2$ if $q=\infty$.  
\end{lem}

\begin{proof}
By \eqref{eq:ug}, taking $s$-th derivative gives 
$$
 u_{g,k}^{(s)} (x):=  \frac{1}{w_{\a,\b}(x)}  \int_x^1 \frac{(y-x)^{s-k-1}}{(s-k-1)!} g(y) w_{\a,\b}(y) dy. 
$$
Taking another $(s-k)$-th order derivatives by the product rule, we conclude that 
$$
  u_{g,k}^{(2s-k)}(x) = -g(x) + \sum_{j=1}^{s-k} \binom{s-k}{j} \frac{d^j}{dx^j}\left(\frac{1}{w_{\a,\b}(x)}\right)
          \int_x^1 \frac{(y-x)^{j-1}}{(j-1)!} g(y) w_{\a,\b}(y) dy.
$$
If $\b = 0$, then $\left| \frac{d^j}{dx^j}\left(\frac{1}{w_{\a,0}(x)}\right)\right| (1-x)^{j-1} = |(-a)_j| (1-x)^{-\a-1}$. Hence, by
$(y-x)^{j-1} \le (1-x)^{j-1}$ for $x \le y\le 1$, it follows that 
$$
\left| u_{g,k}^{(2s-k)}(x) \right| \le |g(x)| + c \sum_{j=1}^{s-k} \frac{1}{(1-x)^{\a+1}}  \int_x^1 |g(y)| w_{\a,0}(y) dy.
$$
If $q = \infty$, it follows immediately that $\|u_{g,k}^{(2s-k)}\|_\infty \le c \|g\|_\infty$. For $1 < q < \infty$, 
by the H\"older inequality and $\phi(x)^{s-k} \le c (1-x)^{1/2}$ as $0\le k \le s-1$, 
\begin{align*}
 &    \int_{-1}^1 \left| \frac{\phi(x)^{s-k}}{(1-x)^{\a+1}} \int_x^1 g(y) w_{\a,0}(y) dy \right|^q w_{\a,0}(x) dx \\ 
 &  \qquad \le c \int_{-1}^1 \int_x^1 | g(y)|^q w_{\a,0}(y) dy \left| \int_x^1w_{\a,0}(t)dt \right|^{q/p} \frac{w_{\a,0}(x)}{(1-x)^{(\a+1/2)q}} dx \\
 &  \qquad \le c \int_{-1}^1 | g(y)|^q w_{\a,\b}(y) dy   \int_{-1}^1 (1-x)^{q/2-1}dx \le c \|g\|_{L^q(w_{\a,0})}^q,
\end{align*}
whereas for $q=1$, exchanging the order of integration shows that 
\begin{align*}
 &    \int_{-1}^1 \left| \frac{\phi(x)^{s-k}}{(1-x)^{\a+1}} \int_x^1 g(y) w_{\a,0}(y) dy \right|  w_{\a,0}(x) dx \\ 
 &  \qquad \le c \int_{-1}^1 | g(y)| w_{\a,0}(y) \int_{-1}^y (1-x)^{\a+1/2}{ (1-x)^{-\a-1}} dx dy \le c \|g\|_{L^1(w_{\a,0})}.
\end{align*}
Putting this together, we have completed proof of \eqref{eq:ug<g} under (a). 

We now prove \eqref{eq:ug<g} under the condition in (b). Since $s =1$, $k$ has to be zero. Since 
$$
  u_{g,0}''(x) = \frac{d}{dx} (w_{\a,\b}(x))^{-1}\int_x^1 g(y) w_{\a,\b}(y)dy - g(x)
$$
and $ \frac{d}{dx} (w_{\a,\b}(x))^{-1} = (\a(1-x)^{-1} - \b (1+x)^{-1} ) (w_{\a,\b}(x))^{-1}$, it is easy to see that our main
task is to establish the inequality 
$$
  \int_{-1}^1 \left| \frac{\phi(x)}{(1-x)w_{\a,\b}(x)} \int_x^1 g(y) w_{\a,\b}(y)dy \right|^q w_{\a,\b}(x) dx \le c \|g\|_{L^q(w_{\a,\b})}^q. 
$$
For $q = \infty$, it is easy to see that we need $\b+1-1/2 \le 0$ or $\b \le -1/2$. For $1 < q < \infty$, we can follow the
proof in (a) and conclude that the inequality holds if $-\b (q-1)> q/2-1$ or $\b < p/2 -1$. This also holds for $q=1$ as can be seen by exchanging the order of the integrals. 
\end{proof}
 
It is worthwhile to mention that if $s > 1$ and $\b \ne 0$, then we apply the product rule on $1/w_{\a,\b}(x) = (1-x)^{-\a}(1+x)^{-\b}$ 
to obtain 
$$
   \left| \frac{d^j}{dx^j}\left(\frac{1}{w_{\a,\b}(x)}\right)\right| (1-x)^{j-1} \le c \sum_{i=0}^j  (1-x)^{-\a+i-1} (1+x)^{-\b-i}, 
$$ 
so that, since the summand in the last expression is independent of $j$, 
$$
\left| u_{g,k}^{(2s-k)}(x) \right| \le |g(x)| + c \sum_{i=0}^{s-k} \frac{(1-x)^{i-1}}{(1-x)^{\a}(1+x)^{\b+i}}  \int_x^1 |g(y)| w_{\a,\b}(y) dy.
$$
Following the same proof as before, it is not difficult to see that we need $\b < (1- s/2)p -1$ for \eqref{eq:ug<g}. Since $\b > -1$, this makes sense only if $s=1$ and $\b < p/2 -1$. 

\begin{thm}
Let $\a,\b> -1$ and $f \in W_p^s(w_{\a,\b})$ for $1\le p < \infty$ and $f\in C^s[-1,1]$ if $p = \infty$. Then 
\begin{equation} \label{eq:simult-appx}
  \| \partial^k f - \partial^k \CV_{n,\t}^{\a-s,\b-s} f\|_{L^p(w_{\a,\b})} \le c \, n^{-s+k}  E_n(f^{(s)})_{L^p(w_{\a,\b})}, \quad 0 \le k \le s,
\end{equation}
holds under either one of the following sets of assumptions:
\begin{enumerate}[\quad (a)]
\item $\t = -1$, with (i) $s \in \NN$,  $\b=0$ and $1 \le p \le \infty$, or (ii) $s=1$, $\b < p/2-1$ for $1< p \le \infty$ 
and $\b \le p/2-1$ if $p=1$;
\item $\t = 1$, with (i) $s \in \NN$, $\a=0$ and $1 \le p \le \infty$, or (ii) $s=1$, $\a < p/2-1$ for $1< p \le \infty$ 
and $\a \le p/2-1$ if $p=1$.  
\end{enumerate}
Furthermore, 
for $p=2$, we can replace \eqref{eq:simult-appx} by
\begin{equation} \label{eq:simult-appx2}
  \| \partial^k f - \partial^k \CS_{n,\t}^{\a-s,\b-s} f\|_{L^2(w_{\a,\b})} \le c \, n^{-s+k}  E_{n-s}(f^{(s)})_{L^2(w_{\a,\b})}, 
  \quad 0 \le k \le s.
\end{equation}
\end{thm}

\begin{proof}
Let us first consider the case $\t = -1$ and write $\CV_{n}^{\a-s,\b-s} f = \CV_{n,-1}^{\a-s,\b-s} f$ for convenience. 
Applying Lemma \ref{lem:int-g} with $v =  \partial^k f - \partial^k \CV_{n}^{\a-s,\b-s} f$ and using \eqref{eq:approxSP-c}, we have
\begin{align*}
\la  \partial^k f - \partial^k \CV_{n}^{\a-s,\b-s} f, g \ra_{\a,\b}  & = \la  f^{(s)} - V_{n}^{\a,\b} f^{(s)},  u_{g,k}^{(s)} \ra_{\a,\b} \\
  & = \la   f^{(s)} - V_{n}^{\a,\b} f^{(s)}, u_{g,k}^{(s)} - V_{\lfloor \frac{n}2\rfloor}^{\a,\b} u_{g,k}^{(s)} \ra_{\a,\b},
\end{align*}
where the second equation follows from the orthogonality of $f^{(s)} - \CV_{n}^{\a,\b} f^{(s)}$ to any polynomials of degree 
at most $n$. By Theorem \ref{thm:Jackson} and \eqref{eq:ug<g}, 
\begin{align*}
\|u_{g,k}^{(s)} - V_{\lfloor \frac{n}2\rfloor}^{\a,\b} u_{g,k}^{(s)}\|_{L^q(w_{\a,\b})} & \le c \, 
    E_{\lfloor \frac{n}2\rfloor}\left(u_{g,k}^{(s)}\right)_{L^q(w_{\a,\b})}\\
    &  \le c \, n^{-s+k} \| \phi^{s-k} u_{g,k}^{(2s-k)}\|_{L^q(w_{\a,\b})} 
        \le c \, n^{-s+k} \| g\|_{L^q(w_{\a,\b})},
\end{align*}
so that, by the H\"older inequality, we conclude that 
\begin{align*}
 \left | \la  \partial^k f - \partial^k \CV_{n}^{\a-s,\b-s} f, g \ra_{\a,\b} \right |  
 & \le c\, \| f^{(s)} - \CV_{n}^{\a,\b} f^{(s)}\|_{L^p(w_{\a,\b})}
      \|u_{g,k}^{(s)} - V_{\lfloor \frac{n}2\rfloor}^{\a,\b} u_{g,k}^{(s)}\|_{L^q(w_{\a,\b})}  \\
  &  \le c \, n^{-s+k} E_n(f)_{L^p(w_{\a,\b})} \| g\|_{L^q(w_{\a,\b})}. 
\end{align*}
Applying the above inequality to the following expression of the $L^p$ norm
\begin{align*}
  \| \partial^k f - \partial^k \CV_{n}^{\a-s,\b-s} f\|_{L^p(w_{\a,\b})} = \sup_{\|g\|_{L^q(w_{\a,\b})} \ne 0} 
      \frac{ \la  \partial^k f - \partial^k \CV_{n}^{\a-s,\b-s} f, g \ra_{\a,\b}}{\|g\|_{L^q(w_{\a,\b})}} 
\end{align*}
completes the proof \eqref{eq:simult-appx} for $\t = -1$. 

The proof for $\t =1$ follows similarly with $u_{g,k}$ in \eqref{eq:ug} replaced by 
\begin{equation*}
  u_{g,k}(x):= \int_{x}^1 \frac{(t-x)^{s-1}}{(s-1)!} \frac{1}{w_{\a,\b}(t)}  \int_{-1}^t \frac{(t-y)^{s-k-1}}{(s-k-1)!} g(y) w_{\a,\b}(y) dy dt 
\end{equation*}
for $0 \le k \le s-1$. The proof follows along the same line and we omit the details. 
\end{proof}
 
Evidently, Theorem \ref{thm:1.2} follows from this theorem with either $\a = 0$ or $\b =0$.

\bigskip
\noindent
{\bf Acknowledgement.} The author thanks Danny Leviatan for his careful readings and corrections.

\end{document}